\documentclass[11pt,leqno]{article}

\usepackage[active]{srcltx}
\usepackage{kerkis}
\usepackage{amsfonts,latexsym,amsmath,amssymb,amsthm}
\usepackage{fullpage}
\newtheorem{theorem}{Theorem}[section]
\newtheorem{lemma}[theorem]{Lemma}

\newtheorem{corollary}[theorem]{Corollary}
\newtheorem{remark}[theorem]{Remark}

\theoremstyle{definition}
\newtheorem{definition}[theorem]{Definition}

\theoremstyle{remark}

\newtheorem*{note*}{Note}
\numberwithin{equation}{section}
\makeatletter

\newcommand{\rank}{\mathop{\operator@font rank}}
\newcommand{\conv}{\mathop{\operator@font conv}}
\newcommand{\vol}{\mathrm{vol}}

\newcommand{\onetagright}{\tagsleft@false}
\makeatother
\newcommand{\ls}{\leqslant}
\newcommand{\gr}{\geqslant}

\newcommand{\set}[1]{\left\{#1\right\}}
\newcommand{\vrad}{{\rm vrad}}

\usepackage{color}

\usepackage{hyperref}

\begin{document}
%\small

\title{\bf On a multi-integral norm defined by weighted sums of log-concave random vectors}

\medskip

\author{Nikos Skarmogiannis}

\date{}

\maketitle

\begin{abstract}
Let $C$ and $K$ be centrally symmetric convex bodies in ${\mathbb R}^n$.
We show that if $C$ is isotropic then
\begin{equation*}\|{\bf t}\|_{C^s,K}=\int_{C}\cdots\int_{C}\Big\|\sum_{j=1}^st_jx_j\Big\|_K\,dx_1\cdots dx_s
\ls c_1L_C(\log n)^5\,\sqrt{n}M(K)\|{\bf t}\|_2\end{equation*}
for every $s\gr 1$ and ${\bf t}=(t_1,\ldots ,t_s)\in {\mathbb R}^s$, where $L_C$ is the isotropic constant of $C$ and $M(K):=\int_{S^{n-1}}\|\xi\|_Kd\sigma (\xi)$. This reduces a question of V.~Milman to the problem of estimating
from above the parameter $M(K)$ of an isotropic convex body. The proof is based on an observation that combines results of
Eldan, Lehec and Klartag on the slicing problem: If $\mu $ is an isotropic log-concave probability measure on ${\mathbb R}^n$
then, for any centrally symmetric convex body $K$ in ${\mathbb R}^n$ we have that
$$I_1(\mu ,K):=\int_{{\mathbb R}^n}\|x\|_K\,d\mu(x)\ls c_2\sqrt{n}(\log n)^5\,M(K).$$
We illustrate the use of this inequality with further applications.
\end{abstract}

%%%%%%%%%%%%%%%%%%%%%%%%%%%%%%%%%%%%%%%%%%%%%%%%%%%%%%%%%%%%%%%%%%%%%%%%%%%%%%%%%%%%%%%%%%%%%%%%%%%%%%%%%%%%%%%%%%%%%%%%%%%%%%%%%%%%%%%%
\section{Introduction}\label{sec:intro}
%%%%%%%%%%%%%%%%%%%%%%%%%%%%%%%%%%%%%%%%%%%%%%%%%%%%%%%%%%%%%%%%%%%%%%%%%%%%%%%%%%%%%%%%%%%%%%%%%%%%%%%%%%%%%%%%%%%%%%%%%%%%%%%%%%%%%%%%

Let $K$ be a centrally symmetric convex body in ${\mathbb R}^n$. For any $s$-tuple ${\cal C}=(C_1,\ldots ,C_s)$ of centrally
symmetric convex bodies $C_j$ of volume $1$ in ${\mathbb R}^n$, consider the norm on ${\mathbb R}^s$, defined by
\begin{equation*}\|{\bf t}\|_{{\cal C},K}=\int_{C_1}\cdots\int_{C_s}\Big\|\sum_{j=1}^st_jx_j\Big\|_K\,dx_s\cdots dx_1\end{equation*}
where ${\bf t}=(t_1,\ldots ,t_s)$. If ${\cal C}=(C,\ldots ,C)$ then we write $\|{\bf t}\|_{C^s,K}$ instead of $\|{\bf t}\|_{{\cal C},K}$.
A question posed by V.~Milman is to determine if, in the case $C=K$, one has that $\|\cdot\|_{K^s,K}$ is equivalent to the
standard Euclidean norm up to a term which is logarithmic in the dimension, and in particular, if under some cotype condition
on the norm induced by $K$ to ${\mathbb R}^n$ one has equivalence between $\|\cdot\|_{K^s,K}$ and the Euclidean norm.

This question was studied by Bourgain, Meyer, V.~Milman and Pajor. For simplicity let us assume
that $\vol_n(K)=1$ (this is only a matter of normalization). It was proved in \cite{BMMP} that
\begin{equation*}\|{\bf t}\|_{{\cal C},K}\gr c\sqrt{s}\Big (\prod_{j=1}^s|t_j|\Big)^{1/s}\end{equation*}
where $c>0$ is an absolute constant. Later, Gluskin and V.~Milman obtained a better lower bound in \cite{Gluskin-VMilman-2004}, in fact
working in a more general context: Let $A_1,\ldots ,A_s$ be measurable sets of volume $1$ in ${\mathbb R}^n$ and $K$ be a star body
of volume $1$ in ${\mathbb R}^n$ with $0\in {\rm int}(K)$. Then,
\begin{equation}\label{eq:GM-basic}\|{\bf t}\|_{{\cal A},K}\gr c\,\|{\bf t}\|_2\end{equation}
for all ${\bf t}=(t_1,\ldots ,t_s)\in {\mathbb R}^s$. Their argument was based on the Brascamp-Lieb-Luttinger inequality (see also \cite[Chapter~4]{AGA-book-2}).

We are mainly interested in upper bounds for the quantity $\|{\bf t}\|_{C^s,K}$. Since $\|{\bf t}\|_{C^s,K}=\|{\bf t}\|_{(TC)^s,TK}$
for any $T\in SL(n)$, we may restrict our attention to the case where $C$ is isotropic
(see Section~\ref{sec:background} for the definition and background information). In fact, we are particularly
interested in the case where $C$ is isotropic and $K=C$, which corresponds to V.~Milman's original question.

For any centered log-concave probability measure $\mu $ on ${\mathbb R}^n$ and any centrally symmetric convex body $K$
in ${\mathbb R}^n$, consider the parameter
\begin{equation}\label{eq:I-definition}I_1(\mu ,K):=\int_{{\mathbb R}^n}\|x\|_Kd\mu (x).\end{equation}
Assume that $C$ is isotropic. Chasapis, Giannopoulos and the author observed in \cite{Chasapis-Giannopoulos-Skarmogiannis-2020}
that one may write
\begin{equation}\label{eq:basic-identity}\|{\bf t}\|_{C^s,K}=\|{\bf t}\|_2L_C\, I_1(\mu_{{\bf t}},K),\end{equation}
where $\mu_{{\bf t}}$ is an isotropic, compactly supported log-concave probability measure depending on ${\bf t}$.
One may easily check that if $\mu $ is an isotropic log-concave probability measure
on ${\mathbb R}^n$ and $K$ is a centrally symmetric convex body in ${\mathbb R}^n$ then
\begin{equation*}\int_{O(n)}I_1(\mu ,U(K))\,d\nu (U) =\int_{{\mathbb R}^n}\int_{O(n)}\|x\|_{U(K)}d\nu (U)\,d\mu (x)=
M(K)\int_{{\mathbb R}^n}\|x\|_2d\mu (x)\approx \sqrt{n}M(K)\end{equation*}
where
\begin{equation*}M(K):=\int_{S^{n-1}}\|\xi\|_Kd\sigma (\xi)\end{equation*}
and $\nu ,\sigma $ denote the Haar probability measures on $O(n)$ and $S^{n-1}$ respectively. It follows that
\begin{equation}\label{eq:intro-1}\int_{O(n)}\|{\bf t}\|_{U(C)^s,K}\,d\nu (U)\approx (L_C\,\sqrt{n}M(K))\,\|{\bf t}\|_2.\end{equation}
Therefore, one might hope to obtain a quantity of the order of $L_C\,\sqrt{n}M(K)\,\|{\bf t}\|_2$ as an upper estimate for $\|{\bf t}\|_{C^s,K}$.
Using an argument which goes back to Bourgain (also, employing Paouris' inequality
and Talagrand's comparison theorem) Chasapis, Giannopoulos and the author proved in
\cite{Chasapis-Giannopoulos-Skarmogiannis-2020} that if $C$ is an isotropic centrally symmetric convex body in $\mathbb{R}^n$ and $K$ is
a centrally symmetric convex body in $\mathbb{R}^n$ then
\begin{equation*}
\|{\bf t}\|_{C^s,K}\ls c\,\Big (L_C\max\Big\{ \sqrt[4]{n},\sqrt{\log (1+s)}\Big\}\Big )\,\sqrt{n}M(K)\|{\bf t}\|_2
\end{equation*}
for every ${\bf t}=(t_1,\ldots,t_s)\in \mathbb{R}^s$, where $c>0$ is an absolute constant. Using the basic identity
\eqref{eq:basic-identity} they also obtained an alternative proof of \eqref{eq:GM-basic}.

In this note we provide a poly-logarithmic in $n$, and independent from $s$, upper bound for $\|{\bf t}\|_{C^s,K}$.

\begin{theorem}\label{th:intro-1}
Let $C$ be an isotropic centrally symmetric convex body in $\mathbb{R}^n$ and $K$ be a centrally symmetric convex body in $\mathbb{R}^n$. Then,
\begin{equation*}
\|{\bf t}\|_{C^s,K}\ls c_1L_C(\log n)^5\,\sqrt{n}M(K)\|{\bf t}\|_2
\end{equation*}
for every ${\bf t}=(t_1,\ldots,t_s)\in \mathbb{R}^s$, where $c_1>0$ is an absolute constant.
\end{theorem}

In the case $C=K$, Theorem~\ref{th:intro-1} and \eqref{eq:GM-basic} show that, for any centrally symmetric convex
body $K$ of volume $1$ in ${\mathbb R}^n$,
\begin{equation*}
c_1\|{\bf t}\|_2\ls \|{\bf t}\|_{K^s,K}\ls c_2(\log n)^5\,\sqrt{n}M(K^{\ast})L_{K^{\ast}}\|{\bf t}\|_2
\end{equation*}
for every $s\gr 1$ and every ${\bf t}=(t_1,\ldots,t_s)\in \mathbb{R}^s$, where $c_i>0$ are absolute constants and
$K^{\ast}$ is an isotropic linear image of $K$ (note that $L_{K^{\ast}}=L_K$). Thus, we have a reduction of V.~Milman's question to
the problem of estimating the parameter $M(K^{\ast})$ for an isotropic centrally symmetric convex body $K^{\ast}$
in ${\mathbb R}^n$. One may hope that $\sqrt{n}M(K^{\ast})L_{K^{\ast}}\ls c_3(\log n)^b$ for some absolute constant $b>0$,
which would completely settle the problem, modulo the best possible exponent of $\log n$. However, the latter problem
remains open; see Remark~\ref{rem:M-problem} for a brief review of what is known.

The source of our improved estimate in Theorem~\ref{th:intro-1} is the next result which we believe is of independent interest.

\begin{theorem}\label{th:intro-2}Let $\mu $ be an isotropic log-concave probability measure on ${\mathbb R}^n$.
For any centrally symmetric convex body $K$ in ${\mathbb R}^n$ we have that
$$I_1(\mu ,K)\ls c_2\sqrt{n}(\log n)^5\,M(K)$$
where $c_2>0$ is an absolute constant.
\end{theorem}

Theorem~\ref{th:intro-2} follows from the recent developments on the Kannan-Lov\'{a}sz-Simonovits conjecture and the
slicing problem. We simply combine the recent estimate of Klartag and Lehec \cite{Klartag-Lehec-2022} with previous
results of Eldan \cite{Eldan-2013} and, in particular, with a theorem of Eldan and Lehec from \cite{Eldan-Lehec-2014}
which now leads to this strong and general estimate for $I_1(\mu ,K)$. We explain the details in Section~\ref{sec:EL}.
In Section~\ref{sec:upper-bound} we combine Theorem~\ref{th:intro-2} with the approach and method that was
introduced in \cite{Chasapis-Giannopoulos-Skarmogiannis-2020} to obtain Theorem~\ref{th:intro-1}.

\smallskip

In Section~\ref{sec:unit-ball} we make some observations on the geometry of the unit ball ${\mathcal B}_s:={\mathcal B}_s(C^s,K)$
of the norm $\|\cdot\|_{C^s,K}$. By the symmetry of $C$ we easily check that ${\mathcal B}_s$ is an unconditional $1$-symmetric convex
body in ${\mathbb R}^s$. Therefore, one can determine the volume radius of ${\mathcal B}_s$ up to a $\sqrt{\log s}$-factor: we have
$$\frac{c}{\sqrt{\log s}}M({\mathcal B}_s)\ls \frac{1}{{\rm vrad}({\mathcal B}_s)}\ls M({\mathcal B}_s)$$
for any $s\gr 1$, where $c>0$ is an absolute constant. In the case $s\gr c_1n$, where $c_1>1$ is an absolute constant,
a result of Bourgain, Meyer, V.~Milman and Pajor from \cite{BMMP} is equivalent to the assertion
\begin{equation*}M({\mathcal B}_s)\gr c_2L_C\sqrt{n}M(K)\end{equation*}
for some absolute constant $c_2>0$. Under the same assumption, that $s\gr c_1n$, we obtain an upper bound of the same
order, and hence determine $M({\mathcal B}_s)$.

\begin{theorem}\label{th:intro-3}There exist absolute constants $c_i>0$ such that: If $C$ is an isotropic centrally symmetric convex
body in ${\mathbb R}^n$ then, for any centrally symmetric convex body $K$ in ${\mathbb R}^n$ and any $s\gr c_1n$,
$$c_2L_C\sqrt{n}M(K)\ls M({\mathcal B}_s)\ls c_3L_C\sqrt{n}M(K)$$
where ${\mathcal B}_s$ is the unit ball of the norm $\|\cdot\|_{C^s,K}$ on ${\mathbb R}^s$.
\end{theorem}

Combining Theorem~\ref{th:intro-3} with Theorem~\ref{th:intro-1} we see that if $s\gr c_1n$
then the critical dimension $k({\mathcal B}_s)$ of ${\mathcal B}_s$ (see \cite[Section~5.3]{AGA-book}) is close to $s$; more precisely,
$$k({\mathcal B}_s)\gr c_4s/(\log n)^{10}.$$
Then, using a well-known deviation inequality for Lipschitz functions on the Euclidean sphere of ${\mathbb R}^s$
we get:

\begin{corollary}\label{cor:intro-4}There exist absolute constants $c_i>0$ such that: If $C$ is an isotropic centrally symmetric convex
body in ${\mathbb R}^n$ then for any centrally symmetric convex body $K$ in ${\mathbb R}^n$ and any $s\gr c_1n$, an $s$-tuple
${\bf t}=(t_1,\ldots ,t_s)$ with $\|{\bf t}\|_2=1$
satisfies
$$c_2L_C\sqrt{n}M(K)\|{\bf t}\|_2\ls \|{\bf t}\|_{{\cal B}_s}\ls c_3L_C\sqrt{n}M(K)\|{\bf t}\|_2$$
with probability greater than $1-\exp (-c_4s/(\log n)^{10})$.
\end{corollary}

We believe that Theorem~\ref{th:intro-2} is a useful result and in Section~\ref{sec:reduction} we
illustrate its use with some additional applications. In the language of convex bodies, Theorem~\ref{th:intro-2}
has the following almost immediate reformulation.

\begin{theorem}\label{th:intro-5}Let $C$ and $K$ be two centrally symmetric convex bodies in ${\mathbb R}^n$.
There exists $T\in SL(n)$ such that
\begin{equation*}\frac{1}{\vol_n(C)}\int_C\|x\|_{TK}dx\ls c(\log n)^{10}\left(\frac{\vol_n(C)}{\vol_n(K)}\right)^{1/n}\end{equation*}
where $c>0$ is an absolute constant.
\end{theorem}

The inequality of Theorem~\ref{th:intro-5} is a reverse form of the well-known fact (see \cite[Section~2.2]{VMilman-Pajor-1989})
that for any pair of centrally symmetric convex bodies $K$ and $C$ in ${\mathbb R}^n$ one has
\begin{equation*}\frac{1}{\vol_n(C)}\int_C\|x\|_{K}dx\gr \frac{n}{n+1}\left(\frac{\vol_n(C)}{\vol_n(K)}\right)^{1/n}.\end{equation*}
It generalizes the $MM^{\ast}$-estimate (see \cite[Theorem~6.5.1]{AGA-book}, more precisely \cite[Corollary~6.5.3]{AGA-book})
which corresponds to the case $C=B_2^n$.

Our second application is an affirmative answer to a question from \cite{Giannopoulos-Paouris-Vritsiou-2012}
where Giannopoulos, Paouris and Vritsiou presented a reduction of the slicing problem to the study of the parameter
$$I_1(K,Z_q^{\circ }(K))=\int_K\|\langle \cdot ,x\rangle\|_{L_q(K)}dx$$
(where $\{Z_q(K)\}_{q\gr 1}$ is the family of $L_q$-centroid bodies of $K$).
It was shown in \cite{Giannopoulos-Paouris-Vritsiou-2012} that if (for some $1/2\ls s< 1$) one had an upper bound of the form
\begin{equation}\label{eq:I-wanted}I_1(K,Z_q^{\circ }(K))\ls c_1q^s\sqrt{n}L_K^2\end{equation}
for every isotropic convex body $K$ in ${\mathbb R}^n$ and all $1\ls q\ls n$, then it would follow that
$$L_n\ls \frac{c_2\sqrt[4]{n}\log\! n}{q^{\frac{1-s}{2}}},$$ where $L_n:=\max\{ L_K:K\;\hbox{is an
isotropic convex body in}\,{\mathbb R}^n\}$. However, to the best of our knowledge, it is only known
that \eqref{eq:I-wanted} holds true with $s=1$, which via the reduction of \cite{Giannopoulos-Paouris-Vritsiou-2012} resulted in an estimate of the form
$L_n=O(\sqrt[4]{n}(\log n)^b)$. Using Theorem~\ref{th:intro-2} we confirm in Section~\ref{sec:reduction} that \eqref{eq:I-wanted} holds true
with $s=1/2$ up to some factors which are logarithmic in $n$.

\begin{theorem}\label{th:intro-6}Let $K$ be  a convex body of volume $1$ in ${\mathbb R}^n$ with center of mass at the origin.
Then, for every $1\ls q\ls n$ we have that
$$I_1(K,Z_q^{\circ }(K))\ls c\sqrt{qn}(\log n)^7L_K^2$$
where $c>0$ is an absolute constant.
\end{theorem}

Note that Theorem~\ref{th:intro-6} establishes \eqref{eq:I-wanted} with $s=\frac{1}{2}$ (up to a factor which is logarithmic
in $n$) in accordance with the fact that the reduction of the slicing problem in \cite{Giannopoulos-Paouris-Vritsiou-2012}
with $q\approx n$ then leads to a poly-logarithmic
in $n$ upper bound for $L_n$. Of course, the order has been now reversed since in the proof of Theorem~\ref{th:intro-6}
we use the results of Eldan, Lehec and Klartag as well as E.~Milman's bounds for the mean width of the
$L_q$-centroid bodies $Z_q(K)$. However, we believe that this application illustrates, once again,
the strength and possible uses of Theorem~\ref{th:intro-2}.

%%%%%%%%%%%%%%%%%%%%%%%%%%%%%%%%%%%%%%%%%%%%%%%%%%%%%%%%%%%%%%%%%%%%%%%%%%%%%%%%%%%%%%%%%%%%%%%%%%%%%%%%%%%%%%%%%%%%%%%%%%%%%%%%%%%%%%%%%%%%%%%
\section{Notation and background information}\label{sec:background}
%%%%%%%%%%%%%%%%%%%%%%%%%%%%%%%%%%%%%%%%%%%%%%%%%%%%%%%%%%%%%%%%%%%%%%%%%%%%%%%%%%%%%%%%%%%%%%%%%%%%%%%%%%%%%%%%%%%%%%%%%%%%%%%%%%%%%%%%%%%%%%%

In this section we introduce notation and terminology that we use throughout this work, and provide background
information on isotropic convex bodies. We write $\langle\cdot ,\cdot\rangle $ for the standard inner product in ${\mathbb R}^n$ and denote the
Euclidean norm by $\|\cdot \|_2$. In what follows, $B_2^n$ is the Euclidean unit ball, $S^{n-1}$ is the unit sphere, and $\sigma $ is the
rotationally invariant probability measure on $S^{n-1}$. Lebesgue measure in ${\mathbb R}^n$ is denoted by ${\rm vol}_n$.
The Grassmann manifold $G_{n,k}$ of all $k$-dimensional subspaces of ${\mathbb R}^n$ is equipped with the Haar probability
measure $\nu_{n,k}$. For every $1\ls k\ls n-1$ and $E\in G_{n,k}$ we write $P_E$ for the orthogonal projection from $\mathbb R^{n}$
onto $E$. The letters $c, c^{\prime },c_j,c_j^{\prime }$ etc. denote absolute positive constants whose value may change from line to line.
Sometimes we might even relax our notation: $a\lesssim b$ will then mean ``$a\ls cb$ for some (suitable) absolute constant $c>0$'',
and $a \approx b$ will stand for ``$a \lesssim b \land a \gtrsim b$". If $A, B$ are sets, $A \approx B$
will similarly state that $c_1A\subseteq B \subseteq c_2 A$ for some absolute constants $c_1,c_2>0$.

A convex body in ${\mathbb R}^n$ is a compact convex set $C\subset {\mathbb R}^n$ with non-empty interior. We say that $C$ is centrally symmetric
if $-C=C$, and that $C$ is centered if its barycenter $\frac{1}{\vol_n(C)}\int_Cx\,dx $ is at the origin. We say that $C$ is unconditional with respect to the standard orthonormal basis $\{e_1,\ldots ,e_n\}$ of ${\mathbb R}^n$ if $x=(x_1,\ldots ,x_n)\in C$ implies that $(\epsilon_1x_1,\ldots ,\epsilon_nx_n)\in C$ for any choice
of signs $\epsilon_i\in\{ -1,1\}$. This is equivalent to the fact that, for any choice of signs $\{\epsilon_i\}_{i=1}^n$
and scalars $t_i$, we have
\begin{equation*}\Big\|\sum_{i=1}^n\epsilon_it_ie_i\Big\| =\Big\|\sum_{i=1}^nt_ie_i\Big\|.\end{equation*}
A centrally symmetric convex body $C$ in ${\mathbb R}^n$ is called $1$-symmetric if the norm induced by $K$ to ${\mathbb R}^n$
is unconditional and permutation invariant, i.e. 
\begin{equation*}\Big\|\sum_{i=1}^n\epsilon_it_ie_{\sigma (i)}\Big \|=\Big\|\sum_{i=1}^nt_ie_i\Big\|\end{equation*}
for any choice of scalars $t_i$, any choice of signs $\{\epsilon_i\}_{i=1}^n$ and any permutation $\sigma $ of $\{1,\ldots ,n\}$.
For any $1\ls p\ls\infty$, the unit ball $B_p^n$ of $\ell_p^n=({\mathbb R}^n,\|\cdot\|_p)$ is a $1$-symmetric convex body.

The radius $R(C)$ of a convex body $C$ is the smallest $R>0$ such that $C\subseteq RB_2^n$.
The volume radius of $C$ is the quantity ${\rm vrad}(C)=\left (\vol_n(C)/\vol_n(B_2^n)\right )^{1/n}$.
Integration in polar coordinates shows that if the origin is an interior point of $C$ then the volume radius of $C$ can be expressed as
\begin{equation*}{\rm vrad}(C)=\Big (\int_{S^{n-1}}\|\xi\|_C^{-n}\,d\sigma (\xi )\Big)^{1/n}\end{equation*}
where $\|\xi\|_C=\inf\{ t>0:\xi\in tC\}$. We also consider the parameter
\begin{equation*}M(C)=\int_{S^{n-1}}\|\xi\|_Cd\sigma (\xi ).\end{equation*}
The support function of $C$ is defined by $h_C(y):=\max \bigl\{\langle x,y\rangle :x\in C\bigr\}$, and
the mean width of $C$ is the average
\begin{equation*}w(C):=\int_{S^{n-1}}h_C(\xi)\,d\sigma (\xi)\end{equation*}
of $h_C$ on $S^{n-1}$. For notational convenience we write $\overline{C}$ for
the homothetic image of volume $1$ of a convex body $C\subseteq
\mathbb R^n$, i.e. $\overline{C}:= \vol_n(C)^{-1/n}C$.

The polar body $C^{\circ }$ of a centrally symmetric convex body $C$ in ${\mathbb R}^n$ is defined by
\begin{equation*}
C^{\circ}:=\bigl\{y\in {\mathbb R}^n: \langle x,y\rangle \ls 1\;\hbox{for all}\; x\in C\bigr\}.
\end{equation*}
The Blaschke-Santal\'{o} inequality states that ${\rm vol}_n(C){\rm vol}_n(C^{\circ })\ls {\rm vol}_n(B_2^n)^2$,
with equality if and only if $C$ is an ellipsoid. The reverse Santal\'{o} inequality of Bourgain and V. Milman
asserts that there exists an absolute constant $c>0$ such
that, conversely,
\begin{equation*}\Big ({\rm vol}_n(C){\rm vol}_n(C^{\circ })\Big )^{1/n}\gr c\,\vol_n(B_2^n)^{2/n}\approx 1/n.\end{equation*}
A convex body $C$ in ${\mathbb R}^n$ is called isotropic if it has volume $1$, it is centered
and its inertia matrix is a multiple of the identity matrix: there exists a constant $L_C >0$ such that
\begin{equation}\label{isotropic-condition}\|\langle \cdot ,\xi\rangle\|_{L_2(C)}^2:=\int_C\langle x,\xi\rangle^2dx =L_C^2\end{equation}
for all $\xi\in S^{n-1}$. The hyperplane conjecture asks if there exists an absolute constant $A>0$ such that
\begin{equation}\label{HypCon}L_n:= \max\{ L_C:C\ \hbox{is isotropic in}\ {\mathbb R}^n\}\ls A\end{equation}
for all $n\gr 1$. Bourgain proved in \cite{Bourgain-1991} that $L_n\ls c\sqrt[4]{n}\log\! n$; later, Klartag \cite{Klartag-2006}
improved this bound to $L_n\ls c\sqrt[4]{n}$. In a breakthrough work, Chen \cite{C} proved that for any $\epsilon >0$
there exists $n_0(\epsilon )\in {\mathbb N}$ such that $L_n\ls n^{\epsilon}$ for every $n\gr n_0(\epsilon )$. Very recently,
Klartag and Lehec \cite{Klartag-Lehec-2022} showed that the hyperplane conjecture and the stronger Kannan-Lov\'{a}sz-Simonovits
isoperimetric conjecture hold true up to a factor that is poly-logarithmic in the dimension; more precisely, they achieved
the bound $L_n\ls c(\log n)^4$, where $c>0$ is an absolute constant. This development will be discussed in Section~\ref{sec:EL}
and, together with previous work of Eldan and Lehec, is the key for our main result.

A Borel measure $\mu$ on $\mathbb R^n$ is called $\log$-concave if $\mu(\lambda
A+(1-\lambda)B) \gr \mu(A)^{\lambda}\mu(B)^{1-\lambda}$ for any compact subsets $A$
and $B$ of ${\mathbb R}^n$ and any $\lambda \in (0,1)$. A function
$f:\mathbb R^n \rightarrow [0,\infty)$ is called $\log$-concave if
its support $\{f>0\}$ is a convex set and the restriction of $\log{f}$ to it is concave.
It is known that if a probability measure $\mu $ is log-concave and $\mu (H)<1$ for every
hyperplane $H$, then $\mu $ has a log-concave density $f_{{\mu }}$. Note that if $C$ is a convex body in
$\mathbb R^n$ then the Brunn-Minkowski inequality implies that
$\mathbf{1}_{C} $ is the density of a $\log$-concave measure.

If $\mu $ is a log-concave measure on ${\mathbb R}^n$ with density $f_{\mu}$, we define the isotropic constant of $\mu $ by
\begin{equation}\label{definition-isotropic}
L_{\mu }:=\left (\frac{\sup_{x\in {\mathbb R}^n} f_{\mu} (x)}{\int_{{\mathbb
R}^n}f_{\mu}(x)dx}\right )^{\frac{1}{n}} [\det \textrm{Cov}(\mu)]^{\frac{1}{2n}}\end{equation}
where $\textrm{Cov}(\mu)$ is the covariance matrix of $\mu$ with entries
\begin{equation}\textrm{Cov}(\mu )_{ij}:=\frac{\int_{{\mathbb R}^n}x_ix_j f_{\mu}
(x)\,dx}{\int_{{\mathbb R}^n} f_{\mu} (x)\,dx}-\frac{\int_{{\mathbb
R}^n}x_i f_{\mu} (x)\,dx}{\int_{{\mathbb R}^n} f_{\mu}
(x)\,dx}\frac{\int_{{\mathbb R}^n}x_j f_{\mu}
(x)\,dx}{\int_{{\mathbb R}^n} f_{\mu} (x)\,dx}.\end{equation} We say
that a $\log $-concave probability measure $\mu $ on ${\mathbb R}^n$
is isotropic if it is centered, i.e. if
\begin{equation}
\int_{\mathbb R^n} \langle x, \xi \rangle d\mu(x) = \int_{\mathbb
R^n} \langle x, \xi \rangle f_{\mu}(x) dx = 0
\end{equation} for all $\xi\in S^{n-1}$, and $\textrm{Cov}(\mu )$ is the identity matrix.

Let $C$ be a centered convex body of volume $1$ in ${\mathbb R}^n$. For every $q\gr 1$,
the $L_q$-centroid body $Z_q(C)$ of $C$ is the centrally symmetric convex body with support function
\begin{equation}\label{Zq-def}h_{Z_q(C)}(y)= \left(\int_C |\langle x,y\rangle|^{q}dx \right)^{1/q}.\end{equation}
From H\"{o}lder's inequality it follows that $Z_1(C)\subseteq Z_p(C)\subseteq Z_q(C)\subseteq Z_{\infty }(C)$ for
all $1\ls p\ls q\ls \infty $, where $Z_{\infty }(C)={\rm conv}\{C,-C\}$.
Using Borell's lemma, one can check that
\begin{equation}\label{eq:Zq-inclusions} Z_q(C)\subseteq c_1\frac{q}{p}Z_p(C)\end{equation}
for all $1\ls p<q$. One can also check that $Z_q(C)\supseteq c_2Z_{\infty }(C)$ for all $q\gr n$.

\smallskip

We refer to the books \cite{AGA-book} and \cite{AGA-book-2} for basic facts from asymptotic convex geometry.
We also refer to \cite{VMilman-Pajor-1989} and \cite{BGVV-book} for more information on isotropic convex bodies
and log-concave probability measures.

%%%%%%%%%%%%%%%%%%%%%%%%%%%%%%%%%%%%%%%%%%%%%%%%%%%%%%%%%%%%%%%%%%%%%%%%%%%%%%%%%%%%%%%%%%%%%%%%%%%%%%%%%%%%%%%%%%%%%%%%%%%%%%%%%%%%%%%%%%%%%%%%%%%%%%%%
\section{The estimate of Eldan and Lehec}\label{sec:EL}
%%%%%%%%%%%%%%%%%%%%%%%%%%%%%%%%%%%%%%%%%%%%%%%%%%%%%%%%%%%%%%%%%%%%%%%%%%%%%%%%%%%%%%%%%%%%%%%%%%%%%%%%%%%%%%%%%%%%%%%%%%%%%%%%%%%%%%%%%%%%%%%%%%%%%%%%

\begin{definition}\label{def:I-one}\rm Let $\mu $ be a centered log-concave probability measure on ${\mathbb R}^n$.
For any convex body $K$ in ${\mathbb R}^n$ with $0\in {\rm int}(K)$ we define
\begin{equation*}I_1(\mu ,K):=\int_{{\mathbb R}^n}\|x\|_Kd\mu (x).\end{equation*}
\end{definition}

In this section we provide an estimate for $I_1(\mu,K)$ in the case where $\mu$ is isotropic.

\begin{proof}[Proof of Theorem~$\ref{th:intro-2}$]
Let $\mu $ be an isotropic log-concave probability measure on ${\mathbb R}^n$.
Eldan and Lehec \cite{Eldan-Lehec-2014} obtained an upper bound for $I_1(\mu ,K)$ which involves the constant
$$\tau_n^2=\sup_{\mu}\,\sup_{\xi\in S^{n-1}}\,\sum_{i,j=1}^n{\mathbb E}_{\mu}(x_ix_j\langle x,\xi\rangle )^2$$
where the first supremum is over all isotropic log-concave probability measures $\mu $ on ${\mathbb R}^n$.

\begin{theorem}[Eldan-Lehec]\label{th:eldan-lehec}Let $\mu $ be an isotropic log-concave probability measure on ${\mathbb R}^n$.
For any centrally symmetric convex body $K$ in ${\mathbb R}^n$ we have that
$$\int_{{\mathbb R}^n}\|x\|_K\,d\mu(x)\ls c_1\sqrt{\log n}\,\tau_n\int_{{\mathbb R}^n}\|x\|_K\,d\gamma_n(x)$$
where $\gamma_n$ is the standard Gaussian measure on ${\mathbb R}^n$ and $c_1>0$ is an absolute constant.
\end{theorem}

A result of Eldan \cite{Eldan-2013} relates the constant $\tau_n$ with the thin-shell constant
$$\sigma_n=\sup_{\mu}\sqrt{{\rm Var}_{\mu}(|x|)}$$
where the supremum is over all isotropic log-concave probability measures $\mu $ on ${\mathbb R}^n$. Eldan proved that
$$\tau_n^2\ls c_2\sum_{k=1}^n\frac{\sigma_k^2}{k}$$
where $c_2>0$ is an absolute constant. On the other hand, Klartag and Lehec \cite{Klartag-Lehec-2022} have obtained a poly-logarithmic
in $n$ upper bound for $\sigma_n$: their main result states that
$$\sigma_n\ls c_3(\log n)^4$$
where $c_3>0$ is an absolute constant. Combining these estimates, one gets
$$\tau_n^2\ls c_2\sum_{k=1}^n\frac{\sigma_k^2}{k}\ls c_4(\log n)^9.$$
Therefore, the estimate of Eldan and Lehec immediately implies that
$$I_1(\mu ,K):=\int_{{\mathbb R}^n}\|x\|_K\,d\mu(x)\ls c_5(\log n)^5\,\int_{{\mathbb R}^n}\|x\|_K\,d\gamma_n(x)$$
where $c_5>0$ is an absolute constant. Finally, integration in polar coordinates shows that
$$\int_{{\mathbb R}^n}\|x\|_K\,d\gamma_n(x)\approx \sqrt{n}\int_{S^{n-1}}\|\xi\|_K\,d\sigma(\xi )\approx \sqrt{n}M(K)$$
and hence the proof of Theorem~\ref{th:intro-2} is complete.\end{proof}

%%%%%%%%%%%%%%%%%%%%%%%%%%%%%%%%%%%%%%%%%%%%%%%%%%%%%%%%%%%%%%%%%%%%%%%%%%%%%%%%%%%%%%%%%%%%%%%%%%%%%%%%%%%%%%%%%%%%%%%%%%%%%%%%%%%%%%%%%%%%%%%%%%%%%%%
\section{Proof of the upper bound}\label{sec:upper-bound}
%%%%%%%%%%%%%%%%%%%%%%%%%%%%%%%%%%%%%%%%%%%%%%%%%%%%%%%%%%%%%%%%%%%%%%%%%%%%%%%%%%%%%%%%%%%%%%%%%%%%%%%%%%%%%%%%%%%%%%%%%%%%%%%%%%%%%%%%%%%%%%%%%%%%%%%

In this section we provide the proof of Theorem~\ref{th:intro-1}. It combines the approach of \cite{Chasapis-Giannopoulos-Skarmogiannis-2020}
with Theorem~\ref{th:intro-2}.

\begin{proof}[Proof of Theorem~$\ref{th:intro-1}$]Let $C$ be an isotropic centrally symmetric convex body in ${\mathbb R}^n$ and $X_1,\ldots ,X_s$ be independent random vectors, uniformly distributed in $C$.
For any ${\bf t}=(t_1\ldots ,t_s)\in {\mathbb R}^s$ we write $\nu_{{\bf t}}$ for the distribution of the random vector $t_1X_1+\cdots +t_sX_s$.
Since $\|{\bf t}\|_{C^s,K}$ is a norm, we may always assume that $\|{\bf t}\|_2=1$. Note that $\nu_{{\bf t}}$ is an even log-concave probability measure on ${\mathbb R}^n$ (this is a consequence of the Pr\'{e}kopa-Leindler inequality; see \cite{AGA-book}). We write $g_{{\bf t}}$ for the density of $\nu_{{\bf t}}$. Our starting point is the next observation from \cite{Chasapis-Giannopoulos-Skarmogiannis-2020}.

\begin{lemma}\label{lem:identity}
For any ${\bf t}=(t_1\ldots ,t_s)\in {\mathbb R}^s$,
we write $\nu_{{\bf t}}$ for the distribution of the random vector $t_1X_1+\cdots +t_sX_s$. Then,
\begin{equation*}\|{\bf t}\|_{C^s,K}=\int_{{\mathbb R}^n}\|x\|_Kd\nu_{{\bf t}}(x).\end{equation*}
\end{lemma}

It is easily verified that the covariance matrix ${\rm Cov}(\nu_{{\bf t}})$ of $\nu_{{\bf t}}$ is a multiple of the identity: more precisely,
\begin{equation*}{\rm Cov}(\nu_{{\bf t}})=L_C^2\,I_n.\end{equation*}
It follows that the function $f_{{\bf t}}(x)=L_C^ng_{{\bf t}}(L_Cx)$ is the density
of an isotropic log-concave probability measure $\mu_{{\bf t}}$ on ${\mathbb R}^n$. Indeed, we have
\begin{equation*}\int_{{\mathbb R}^n}f_{{\bf t}}(x)x_ix_j\,dx=L_C^n\int_{{\mathbb R}^n}g_{{\bf t}}(L_Cx)x_ix_j\,dx
=L_C^{-2}\int_{{\mathbb R}^n}g_{{\bf t}}(y)y_iy_j\,dy=\delta_{ij}\end{equation*}
for all $1\ls i,j\ls n$.

\begin{note*}It is proved in \cite{Chasapis-Giannopoulos-Skarmogiannis-2020} that if $\|{\bf t}\|_2=1$ then $\|g_{{\bf t}}\|_{\infty }\ls e^n$.
From this inequality we see that
\begin{equation*}L_{\mu_{{\bf t}}}=\|f_{{\bf t}}\|_{\infty }^{\frac{1}{n}}=L_C\|g_{{\bf t}}\|_{\infty }^{\frac{1}{n}}\ls eL_C\end{equation*}
for all ${\bf t}\in {\mathbb R}^s$ with $\|{\bf t}\|_2=1$.
\end{note*}

We compute
\begin{equation*}\|{\bf t}\|_{C^s,K}=\int_{{\mathbb R}^n}\|x\|_K\,d\nu_{{\bf t}}(x)=L_C^{-n}\int_{{\mathbb R}^n}\|x\|_Kf_{{\bf t}}(x/L_C)\,dx
= L_C\int_{{\mathbb R}^n}\|y\|_Kd\mu_{{\bf t}}(y)\end{equation*}
and hence we get
\begin{equation}\label{eq:isotropic-identity}\|{\bf t}\|_{C^s,K}=L_C\, I_1(\mu_{{\bf t}},K)\end{equation}
for all ${\bf t}\in {\mathbb R}^s$ with $\|{\bf t}\|_2=1$. Now, we use Theorem~\ref{th:intro-2} to estimate
$I_1(\mu_{{\bf t}},K)$. As a result, we obtain the upper bound
$$\|{\bf t}\|_{C^s,K}\ls c_1L_C(\log n)^5\,\sqrt{n}M(K)$$
which is the assertion of Theorem~\ref{th:intro-1}.\end{proof}

\begin{remark}\label{rem:M-problem}\rm We mentioned in the introduction that the question to estimate
the parameter $M(K)$ for an isotropic centrally symmetric convex body $K$ in ${\mathbb R}^n$ remains open.
The currently best known estimates are due to Giannopoulos and E.~Milman (see \cite{Giannopoulos-EMilman-2014};
also, \cite{GSTV} for previous work on this question). In view of the recent developments on the slicing
problem, we briefly recall the main computations in their argument and make some slight modifications to get
the next result.

\begin{theorem}[Giannopoulos-E.~Milman]\label{th:GM-1}Let $K$ be an isotropic centrally symmetric convex body in ${\mathbb R}^n$. Then,
\begin{equation}\label{eq:em-01}M(K)\ls \frac{cn^{1/3}(\log n)^{5/3}}{\sqrt{n}L_K}\end{equation}
where $c>0$ is an absolute constant.
\end{theorem}

\begin{proof}It is proved in \cite[Corollary~4.4]{Giannopoulos-EMilman-2014} that for every centrally symmetric convex body $K$ in
${\mathbb R}^n$ with $K \supseteq r B_2^n$, one has
\begin{equation}\label{eq:general-M-symmetric}\sqrt{n} M(K) \ls c_1\sum_{k=1}^{n}  \frac{1}{\sqrt{k}} \min\left\{\frac{1}{r} , \frac{n}{k} \log\Big(e + \frac{n}{k}\Big)\frac{1}{v_{k}^{-}(K)}\right\}\end{equation}
where
\begin{equation*}v^{-}_k(K) := \inf \set{ \vrad(P_E (K)) : E \in G_{n,k}}.\end{equation*}
Now, since $K$ is isotropic we know that $K\supseteq L_KB_2^n$, therefore we may use \eqref{eq:general-M-symmetric} with
$r=L_K$. We also know (see \cite[Prooposition~5.1.15]{BGVV-book}) that
\begin{equation*}\vol (K\cap E^{\perp})^{1/k}\ls c_2\frac{L_k}{L_K}\end{equation*}
for every $E\in G_{n,k}$. Applying the Rogers-Shephard inequality we see that for every $E\in G_{n,k}$ we have
$$\vol (P_E(K))^{1/k}\gr \frac{1}{\vol (K\cap E^{\perp })^{1/k}}\gr  c_3\frac{L_K}{L_k}$$
for every $E\in G_{n,k}$. Therefore, ${\rm vrad}(P_E(K))\gr c_4\sqrt{k}L_K/L_k$, which gives
$$v_k^-(K)\gr c_4\sqrt{k}L_K/L_k.$$ Set $k_n=n^{2/3}(\log n)^{10/3}$. Inserting the above estimates into
\eqref{eq:general-M-symmetric} and using the fact that $L_k=O((\log n)^4)$, we get
\begin{align*}\sqrt{n}M(K) &\ls \frac{c_5}{L_K}\sum_{k=1}^n\frac{1}{\sqrt{k}}\min\left\{ 1,\frac{n(\log n)^5}{k^{3/2}}\right\}\ls
\frac{c_6}{L_K}\left(\sum_{k=1}^{k_n}\frac{1}{\sqrt{k}}+\sum_{k=k_n}^n\frac{n(\log n)^5}{k^2}\right)\\
&\approx \frac{n^{1/3}(\log n)^{5/3}}{L_K}.
\end{align*}
This proves the theorem.\end{proof}
\end{remark}

%%%%%%%%%%%%%%%%%%%%%%%%%%%%%%%%%%%%%%%%%%%%%%%%%%%%%%%%%%%%%%%%%%%%%%%%%%%%%%%%%%%%%%%%%%%%%%%%%%%%%%%%%%%%%%%%%%%%%%%%%%%%%%%%%%%%%%%%
\section{Geometry of the unit ball}\label{sec:unit-ball}
%%%%%%%%%%%%%%%%%%%%%%%%%%%%%%%%%%%%%%%%%%%%%%%%%%%%%%%%%%%%%%%%%%%%%%%%%%%%%%%%%%%%%%%%%%%%%%%%%%%%%%%%%%%%%%%%%%%%%%%%%%%%%%%%%%%%%%%%

Let $C$ and $K$ be centrally symmetric convex bodies in ${\mathbb R}^n$ and assume that $C$ is isotropic.
In this section we fix $C$ and $K$, and write ${\mathcal B}_s:={\mathcal B}_s(C^s,K)\subset {\mathbb R}^s$ for the unit ball of the norm
\begin{equation*}\|{\bf t}\|_{C^s,K}=\int_C\cdots\int_C\Big\|\sum_{j=1}^st_jx_j\Big\|_K\,dx_s\cdots dx_1.\end{equation*}
By the symmetry of $C$ we easily check that $\|\cdot\|_{C^s,K}$ is an unconditional norm. Moreover, it is $1$-symmetric. Therefore,
${\mathcal B}_s$ is a $1$-symmetric convex body in ${\mathbb R}^s$ and hence it is known that it satisfies the $MM^{\ast}$-estimate
\begin{equation}\label{eq:geometry-1}M({\mathcal B}_s)w({\mathcal B}_s)\ls c\sqrt{\log s}\end{equation}
where $c>0$ is an absolute constant (see \cite{Tomczak-book}). We also recall the general bounds
\begin{equation}\label{eq:geometry-2}\frac{1}{M({\mathcal B}_s)}\ls {\rm vrad}({\mathcal B}_s)\ls w({\mathcal B}_s)\end{equation}
that hold for any convex body which contains the origin as an interior point.

\begin{lemma}\label{lem:N-0}For any $s\gr 1$ we have that $$\frac{c_1}{\sqrt{\log s}}M({\mathcal B}_s)\ls\frac{c_2}{w({\mathcal B}_s)}\ls \frac{1}{\sqrt{s}}\Big\|\sum_{i=1}^se_i\Big\|_{{\mathcal B}_s}\approx \frac{1}{{\rm vrad}({\mathcal B}_s)}\ls M({\mathcal B}_s)$$
where $c_1,c_2>0$ are absolute constants.
\end{lemma}

\begin{proof}Since ${\mathcal B}_s$ is $1$-symmetric, we know that $\overline{{\mathcal B}_s}:=|{\mathcal B}_s|^{-1/s}{\mathcal B}_s$ is isotropic.
A well-known result of Bobkov and Nazarov (see \cite[Chapter~4]{BGVV-book}) shows that
$$c_1|{\mathcal B}_s|^{1/s}B_{\infty }^s\subseteq {\mathcal B}_s\subseteq c_2s|{\mathcal B}_s|^{1/s}B_1^s.$$
Equivalently, for all $t_1,\ldots ,t_s\in {\mathbb R}$ we have that
$$c_2^{\prime }|{\mathcal B}_s|^{-1/s}\cdot \frac{1}{s}\sum_{i=1}^s|t_i|\ls \Big\|\sum_{i=1}^st_ie_i\Big\|_{{\mathcal B}_s}
\ls c_1^{\prime }|{\mathcal B}_s|^{-1/s}\cdot \max_{1\ls i\ls s}|t_i|.$$
Choosing $t_1=\cdots =t_s=1$ we get the equivalence
$$\frac{1}{\sqrt{s}}\Big\|\sum_{i=1}^se_i\Big\|_{{\mathcal B}_s}\approx \frac{1}{{\rm vrad}({\mathcal B}_s)}.$$
The remaining assertions of the lemma follow from \eqref{eq:geometry-1} and \eqref{eq:geometry-2}.
\end{proof}

The following theorem is proved in \cite{BMMP}.

\begin{theorem}[Bourgain-Meyer-V.~Milman-Pajor]\label{th:BMMP-2}There exist absolute constants $c_i>0$ such that, if $C$ is an
isotropic centrally symmetric convex body in ${\mathbb R}^n$
then for any centrally symmetric convex body $K$ in ${\mathbb R}^n$ and any $s\gr c_1n$,
$$\int_C\cdots\int_C\int_{\Omega }\Big\|\sum_{j=1}^sg_j(\omega )x_j\Big\|_K\,d\omega \,dx_s\cdots dx_1\gr c_2\sqrt{s}L_C\sqrt{n}M(K).$$
\end{theorem}

\begin{note*}We easily check that
\begin{align*}\int_C\cdots\int_C\int_{\Omega }\Big\|\sum_{j=1}^sg_j(\omega )x_j\Big\|_K\,d\omega \,dx_s\cdots dx_1
&=\int_{\Omega }\Big\|\sum_{j=1}^sg_j(\omega )e_j\Big\|_{C^s,K}d\omega \\
&=\int_{{\mathbb R}^s}\|y\|_{C^s,K}d\gamma_s(y)\approx \sqrt{s}M({\mathcal B}_s).
\end{align*}
Therefore, Theorem~\ref{th:BMMP-2} states that if $s\gr c_1n$ then
\begin{equation}\label{eq:BMMP-low}M({\mathcal B}_s)\gr c_2L_C\sqrt{n}M(K)\end{equation}
for some absolute constant $c_2>0$. We will show that (for the same values of $s$) a reverse inequality is also true.
\end{note*}

\begin{theorem}\label{th:N-1}There exist absolute constants $c_i>0$ such that: If $C$ is an isotropic centrally symmetric convex
body in ${\mathbb R}^n$ then for any centrally symmetric convex body $K$ in ${\mathbb R}^n$ and any $s\gr c_1n$,
$$c_2L_C\sqrt{n}M(K)\ls M({\mathcal B}_s)\ls c_3L_C\sqrt{n}M(K),$$
where ${\mathcal B}_s$ is the unit ball of the norm $\|\cdot\|_{C^s,K}$ on ${\mathbb R}^s$.
\end{theorem}

\begin{proof}We shall use a well-known result of Adamczak, Litvak, Pajor and Tomczak-Jaegermann
from \cite{Adamczak-Litvak-Pajor-Tomczak-2010}: For every $\epsilon>0$ and $t\gr 1$ there exists
a constant $C=C(t,\epsilon)>0$ with the following property: if $s\gr C(t,\epsilon)n$ and $X_1,\ldots ,X_s$ are independent
random vectors in ${\mathbb R}^n$ with the same isotropic log-concave distribution, then we have
$$\Big\|\frac{1}{s}\sum_{j=1}^sX_j\otimes X_j-I_n\Big\|\ls\epsilon $$
with probability greater than $1-\exp(-ct\sqrt{n})$. Choosing $\epsilon=\frac{1}{2}$, $t=1$ and applying this
result we get that there exists an absolute constant $c_1>0$ such that if $s\gr c_1n$ then there exists ${\cal A}\subset C^s$ of measure
$|{\cal A}|>1-e^{-c_2\sqrt{n}}$ such that for any $s$-tuple $(x_1,\ldots ,x_s)\in {\cal A}$ we have that
\begin{equation}\label{eq:slepian-double}\frac{1}{2}sL_C^2\|z\|_2^2\ls \sum_{j=1}^s\langle x_j,z\rangle^2\ls\frac{3}{2}sL_C^2\|z\|_2^2 \end{equation}
for every $z\in {\mathbb R}^n$. Fix $(x_1,\ldots ,x_s)\in {\cal A}$. We have
$$\frac{1}{2}sL_C^2\sum_{j=1}^n\langle z,e_j\rangle^2\ls \sum_{j=1}^s\langle z,x_j\rangle^2\ls\frac{3}{2}sL_C^2\sum_{j=1}^n\langle z,e_j\rangle^2$$
for all $z\in {\mathbb R}^n$. Applying Slepian's comparison principle (see \cite{Vershynin-book} or \cite[Chapter~9]{AGA-book}) for the Gaussian processes
$$X_z:=\Big\langle\sum_{j=1}^sg_jx_j,z\Big\rangle\quad\hbox{and}\quad Y_z:=\sum_{j=1}^sz_jg_j^{\prime}\qquad (z\in K^{\circ})$$
where $g_j,g_j^{\prime}$ are independent standard Gaussian random variables, we see that
\begin{equation}\label{eq:slepian-1}\int_{\Omega }\Big\|\sum_{j=1}^sg_j(\omega )x_j\Big\|_K\,d\omega \ls c_3\sqrt{s}L_C\int_{\Omega }\Big\|\sum_{j=1}^ng_j^{\prime }(\omega )e_j\Big\|_K\,d\omega \ls c_4\sqrt{s}L_C\sqrt{n}M(K)\end{equation}
for every $(x_1,\ldots ,x_s)\in {\cal A}$. On the other hand, if $(x_1,\ldots ,x_s)\notin {\cal A}$ then
we may write
\begin{equation*}\sum_{j=1}^s\langle x_j,z\rangle^2\ls \sum_{j=1}^s\|x_j\|_2^2\|z\|_2^2\ls
sR(C)^2\|z\|_2^2\ls c_5sn^2L_C^2\|z\|_2^2,\end{equation*}
recalling that $R(C)\ls c_6nL_C$ because $C$ is isotropic. Working as before, we obtain the bound
\begin{equation}\label{eq:slepian-2}\int_{\Omega }\Big\|\sum_{j=1}^sg_j(\omega )x_j\Big\|_K\,d\omega
\ls c_7\sqrt{s}nL_C\sqrt{n}M(K)\end{equation}
for every $(x_1,\ldots ,x_s)\notin {\cal A}$. Integrating on ${\cal A}$ and using \eqref{eq:slepian-1} and \eqref{eq:slepian-2} we finally get
\begin{align*}\sqrt{s}M({\mathcal B}_s) &\approx \int_C\cdots\int_C\int_{\Omega }\Big\|\sum_{j=1}^sg_j(\omega )x_j\Big\|_K\,d\omega \,dx_s\cdots dx_1\\
&=\int_{{\cal A}}\int_{\Omega }\Big\|\sum_{j=1}^sg_j(\omega )x_j\Big\|_K\,d\omega
+\int_{C^s\setminus {\cal A}}\int_{\Omega }\Big\|\sum_{j=1}^sg_j(\omega )x_j\Big\|_K\,d\omega \\
&\ls c_8\sqrt{s}\big (1+n\cdot\exp(-c\sqrt{n})\big )L_C\sqrt{n}M(K)\ls c_9 \sqrt{s}L_C\sqrt{n}M(K),
\end{align*}
which shows that
$$M({\mathcal B}_s)\ls c_{10}L_C\sqrt{n}M(K).$$
The proof of \eqref{eq:BMMP-low}, which was already obtained in \cite{BMMP}, follows by a similar (and simpler) argument
if we start from the left-hand side inequality of \eqref{eq:slepian-double} .\end{proof}

\begin{remark}\label{rem:local}\rm Combining Theorem~\ref{th:N-1} with Theorem~\ref{th:intro-1}
we obtain further information in the case $s\gr c_1n$. Note that the critical dimension
of ${\mathcal B}_s$ (see \cite[Section~5.3]{AGA-book}) satisfies $k({\mathcal B}_s)\approx s\big(M/b\big)^2$
where $M=M({\mathcal B}_s)\gtrsim L_C\sqrt{n}M(K)$ by \eqref{eq:BMMP-low} and $b$ is the best constant for which
$\|{\bf t}\|_{C^s,K}\ls b\|{\bf t}\|_2$ for all ${\bf t}\in {\mathbb R}^s$. From Theorem~\ref{th:intro-1} we know
that $b\lesssim c_1L_C(\log n)^5\,\sqrt{n}M(K)$, and hence
$$k({\mathcal B}_s)\gr cs/(\log n)^{10}.$$
A standard application of the deviation inequality for Lipschitz functions on the Euclidean sphere of ${\mathbb R}^s$
(see \cite[Proposition~5.2.4]{AGA-book}) shows that an $s$-tuple ${\bf t}=(t_1,\ldots ,t_s)$ with $\|{\bf t}\|_2=1$
satisfies
$$\frac{1}{2}M({\cal B}_s)\|{\bf t}\|_2\ls \|{\bf t}\|_{{\cal B}_s}\ls \frac{3}{2}M({\cal B}_s)\|{\bf t}\|_2$$
with probability greater than $1-\exp (-c_3s/(\log n)^{10})$ and hence, by Theorem~\ref{th:N-1},
$$\|{\bf t}\|_{{\cal B}_s}\approx L_C\sqrt{n}M(K)\|{\bf t}\|_2$$
with probability greater than $1-\exp (-c_3s/(\log n)^{10})$ on the sphere of ${\mathbb R}^s$.
\end{remark}

%%%%%%%%%%%%%%%%%%%%%%%%%%%%%%%%%%%%%%%%%%%%%%%%%%%%%%%%%%%%%%%%%%%%%%%%%%%%%%%%%%%%%%%%%%%%%%%%%%%%%%%%%%%%%%%%%%%%%%%%%%%%%%%
\section{Further remarks and applications}\label{sec:reduction}
%%%%%%%%%%%%%%%%%%%%%%%%%%%%%%%%%%%%%%%%%%%%%%%%%%%%%%%%%%%%%%%%%%%%%%%%%%%%%%%%%%%%%%%%%%%%%%%%%%%%%%%%%%%%%%%%%%%%%%%%%%%%%%%

We believe that Theorem~\ref{th:intro-2} is a useful result that should find several applications. As a first example, we
prove Theorem~\ref{th:intro-5}: If $C$ and $K$ are two centrally symmetric convex bodies in ${\mathbb R}^n$ then there exists $T\in SL(n)$ such that
\begin{equation*}\frac{1}{\vol_n(C)}\int_C\|x\|_{TK}dx\ls c(\log n)^9\left(\frac{\vol_n(C)}{\vol_n(K)}\right)^{1/n}\end{equation*}
where $c>0$ is an absolute constant.

\begin{proof}[Proof of Theorem~$\ref{th:intro-5}$]By homogeneity we may assume that $\vol_n(C)=1$. We may find $T_1\in SL(n)$ such that
$C_1:=T_1(C)$ is isotropic. Consider the isotropic log-concave probability measure $\nu_C$ with density
$L_C^n{\mathbf 1}_{K/L_C}$. Direct computation shows that
$$\int_{C_1}\|x\|_{T_1T(K)}dx=I_1(\nu_{C_1},T_1T(K))L_C$$
for every $T\in SL(n)$. Using Theorem~\ref{th:intro-2} we get
$$\int_C\|x\|_{TK}dx=\int_{C_1}\|x\|_{T_1T(K)}dx\ls c_1\sqrt{n}(\log n)^5L_C\,M(T_1T(K))$$
for every $T\in SL(n)$, where $c_1>0$ is an absolute constant. From Pisier's inequality (see \cite[Corollary~6.5.3]{AGA-book}) and the
Blaschke-Santal\'{o} inequality we may choose $T$ so that
\begin{align*}M(T_1T(K))&=w((T_1T(K))^{\circ})=w((T_1T)^{-\ast}(K^{\circ}))\ls c_2\sqrt{n}(\log n)\vol_n(K^{\circ})^{1/n}\\
&\ls c_3\sqrt{n}(\log n)\frac{\vol_n(B_2^n)^{2/n}}{\vol_n(K)^{1/n}}\approx \frac{\log n}{\sqrt{n}}\vol_n(K)^{-1/n},
\end{align*}
since $\vol_n(B_2^n)^{2/n}\approx 1/n$. Combining the above we see that there exists $T\in SL(n)$ such that
$$\int_C\|x\|_{TK}dx\ls c(\log n)^6L_C\,\vol_n(K)^{-1/n}$$
and the result follows from the bound $L_C\ls L_n=O((\log n)^4)$.\end{proof}

\smallskip

As a second application we prove Theorem~\ref{th:intro-6} which gives an answer to a question from \cite{Giannopoulos-Paouris-Vritsiou-2012}
regarding the parameter
$$I_1(K,Z_q^{\circ }(K)):=\int_K\|\langle \cdot ,x\rangle \|_{L_q(K)}dx.$$

\begin{proof}[Proof of Theorem~$\ref{th:intro-6}$]Note that $I_1(K,Z_q^{\circ }(K))$  is invariant under invertible linear
transformations of $K$ and hence we may assume that $K$ is isotropic.
As in the proof of Theorem~\ref{th:intro-5}, consider the isotropic log-concave probability measure $\nu_K$ with
density $L_K^n{\mathbf 1}_{K/L_K}$. Direct computation shows that
$$I_1(K,Z_q^{\circ }(K))=I_1(\nu_K,Z_q^{\circ }(K))L_K.$$
Using Theorem~\ref{th:intro-2} we immediately see that
$$I_1(\nu_K,Z_q^{\circ}(K))\ls c_1(\log n)^5\,\int_{{\mathbb R}^n}\|x\|_{Z_q(K)^{\circ}}\,d\gamma_n(x)
\ls c_2\sqrt{n}(\log n)^5w(Z_q(K)).$$
We use E.~Milman's estimates \cite{EMilman-2014} on the mean width $w(Z_q(K))$ of
the $L_q$-centroid bodies $Z_q(K)$ of an isotropic convex body $K$ in ${\mathbb R}^n$.

\begin{theorem}[E.~Milman]\label{th:Emanuel2}Let $K$ be an isotropic convex body in ${\mathbb R}^n$. Then, for all $1\ls q\ls n$ one has
\begin{equation}w(Z_q(K)) \ls c_1\sqrt{q}\log (1+q)\max\left\{\frac{\sqrt{q}\log (1+q)}{\sqrt{n}},1\right\}L_K\end{equation}
where $c_1>0$ is an absolute constant.
\end{theorem}

Combining the above we get
$$I_1(K,Z_q^{\circ }(K))=I_1(\nu_K,Z_q^{\circ }(K))L_K\ls c_1\sqrt{qn}(\log n)^5\log (1+q)\max\left\{\frac{\sqrt{q}\log (1+q)}{\sqrt{n}},1\right\}L_K^2$$
for all $1\ls q\ls n$ and the theorem follows.\end{proof}

%%%%%%%%%%% End of paper body %%%%%%%%%%%%%%%%%%%%%%%%%%%%%%%
\bigskip

\noindent {\bf Acknowledgement.} The author acknowledges support by the Hellenic Foundation for
Research and Innovation (H.F.R.I.) under the ``First Call for H.F.R.I.
Research Projects to support Faculty members and Researchers and
the procurement of high-cost research equipment grant" (Project
Number: 1849).

\bigskip

\small

%\footnotesize

\bibliographystyle{amsplain}

\begin{thebibliography}{100}


\bibitem{Adamczak-Litvak-Pajor-Tomczak-2010}  {\rm R.\ Adamczak, A. E.\ Litvak, A.\ Pajor and N.\ Tomczak-Jaegermann},
\textit{Quantitative estimates of the convergence of the empirical covariance matrix in log-concave ensembles}, J.\ Amer.\ Math.\ Soc.\
{\bf 23} (2010), No.\ 2, 535--561.

\bibitem{AGA-book} {\rm S.\ Artstein-Avidan, A.\ Giannopoulos and V.\ D.\ Milman},
\textit{Asymptotic Geometric Analysis, Vol. I}, Mathematical Surveys and Monographs \textbf{202}, American Mathematical Society, Providence, RI, 2015.
\bibitem{AGA-book-2} {\rm S.\ Artstein-Avidan, A.\ Giannopoulos and V.\ D.\ Milman},
\textit{Asymptotic Geometric Analysis, Vol. II}, Mathematical Surveys and Monographs \textbf{261}, American Mathematical Society, Providence, RI, 2021.
\bibitem{Bourgain-1991} {\rm J.\ Bourgain}, \textit{On the distribution of polynomials on high dimensional convex sets},
Lecture Notes in Mathematics {\bf 1469}, Springer, Berlin (1991), 127--137.
\bibitem{BMMP} {\rm J.\ Bourgain, M.\ Meyer, V.\ D.\ Milman and A.\ Pajor}, \textit{On a geometric inequality}, Geometric
aspects of functional analysis (1986-'87), Lecture Notes in Math., 1317, Springer, Berlin (1988), 271--282.


\bibitem{BGVV-book} {\rm S.\ Brazitikos, A.\ Giannopoulos, P.\ Valettas and B-H.\ Vritsiou}, \textit{Geometry of isotropic
convex bodies}, Mathematical Surveys and Monographs \textbf{196}, American Mathematical Society, Providence, RI, 2014.
\bibitem{Chasapis-Giannopoulos-Skarmogiannis-2020} {\rm G.\ Chasapis, A.\ Giannopoulos and N.\ Skarmogiannis},
\textit{Norms of weighted sums of log-concave random vectors}, Commun. Contemp. Math. {\bf 22} (2020), no. 4, 1950036, 31 pp.
\bibitem{C} {Y.~Chen,} \textit{An almost constant lower bound of the isoperimetric coefficient in the KLS conjecture},
Geom. Funct. Anal. {\bf 31} (2021), 34--61.

\bibitem{Eldan-2013} {\rm R.\ Eldan}, \textit{Thin shell implies spectral gap up to polylog via a stochastic localization scheme},
Geom. Funct. Anal. {\bf 23} (2013), no. 2, 532--569.
\bibitem{Eldan-Lehec-2014} {\rm R.\ Eldan and J.\ Lehec}, \textit{Bounding the norm of a log-concave vector via thin-shell estimates},
Geometric aspects of functional analysis, 107--122, Lecture Notes in Math., 2116, Springer, Cham, 2014.

\bibitem{Giannopoulos-EMilman-2014} {\rm A.\ Giannopoulos and E.\ Milman}, \textit{$M$-estimates for isotropic convex bodies and
their $L_q$-centroid bodies}, in Geom. Aspects of Funct. Analysis, Lecture Notes in Mathematics {\bf 2116} (2014), 159--182.

\bibitem{Giannopoulos-Paouris-Vritsiou-2012} {\rm A.\ Giannopoulos, G.\ Paouris and B-H.\ Vritsiou},
\textit{A remark on the slicing problem}, Journal of Functional Analysis {\bf 262} (2012), 1062--1086.

\bibitem{GSTV} {\rm A.\ Giannopoulos, P.\ Stavrakakis, A.\ Tsolomitis and B-H.\ Vritsiou},
\textit{Geometry of the $L_q$-centroid bodies of an isotropic log-concave measure},
Trans. Amer. Math. Soc. {\bf 367} (2015), 4569--4593.

\bibitem{Gluskin-VMilman-2004} {\rm E.\ D.\ Gluskin and V.\ D.\ Milman}, \textit{Geometric probability and random cotype $2$},
Geometric aspects of functional analysis, 123--138, Lecture Notes in Math., 1850, Springer, Berlin, 2004.
\bibitem{Klartag-2006} {\rm B.\ Klartag}, \textit{ On convex perturbations with a bounded isotropic constant},
Geom.\ Funct.\ Anal.\ {\bf 16} (2006), 1274--1290.
\bibitem{Klartag-Lehec-2022} {\rm B.\ Klartag and J.\ Lehec}, \textit{Bourgain's slicing problem and KLS isoperimetry up to
polylog}, Preprint.
\bibitem{EMilman-2014} {\rm E.\ Milman}, \textit{On the mean width of isotropic convex bodies and
their associated $L_p$-centroid bodies}, Int. Math. Res. Not. IMRN (2015), no. 11, 3408--3423.

\bibitem{VMilman-Pajor-1989} {\rm V.\ D.\ Milman and A.\ Pajor}, \textit{Isotropic position and inertia ellipsoids and zonoids of the unit
ball of a normed $n$-dimensional space}, Lecture Notes in Mathematics {\bf 1376}, Springer, Berlin (1989), 64--104.

\bibitem{Tomczak-book} {\rm N. Tomczak-Jaegermann}, {\sl Banach-Mazur Distances and Finite Dimensional Operator Ideals},
Pitman Monographs {\bf 38} (1989), Pitman, London.

\bibitem{Vershynin-book} {\rm R.\ Vershynin}, \textit{High-Dimensional Probability: An Introduction with Applications in Data Science},
Cambridge Series in Statistical and Probabilistic Mathematics {\bf 47}, Cambridge University Press, Cambridge, 2018.


\end{thebibliography}

\normalsize

%\small

\bigskip

\noindent{\bf Keywords:} convex bodies, log-concave probability measures, weighted sums of random vectors, isotropic position.
\\
\thanks{\noindent {\bf 2010 MSC:} Primary 46B06; Secondary 52A23, 52A40, 60D05.}

\bigskip

\bigskip

\noindent \textsc{Nikos \ Skarmogiannis}: Department of
Mathematics, National and Kapodistrian University of Athens, Panepistimioupolis 157-84,
Athens, Greece.

\smallskip

\noindent \textit{E-mail:} \texttt{nikskar@math.uoa.gr}

\end{document}